\documentclass[prd,onecolumn,nopacs,nofootinbib]{revtex4}

\usepackage{amsmath}
\usepackage{graphicx}
\usepackage{amsfonts}
\usepackage{latexsym}
\usepackage{bbold}
\usepackage{wasysym}
\usepackage{calligra}
\usepackage{float}
\usepackage{ulem}
\usepackage{inputenc}
\usepackage{xspace}
\usepackage{url}
\usepackage{epstopdf}
\usepackage{tikz}
\usepackage{amsthm}

\usepackage[framemethod=TikZ]{mdframed}
\usepackage{lipsum}

\newcommand{\be}{\begin{equation}}
\newcommand{\ee}{\end{equation}}
\newcommand{\beq} {\begin{equation}}
\newcommand{\eeq} {\end{equation}}
\newcommand{\ba}{\begin{eqnarray}}
\newcommand{\ea}{\end{eqnarray}}

\newtheorem{mydef}{Definition}
\newtheorem*{corollary*}{Corollary}

\newtheorem{theo}{Theorem}

\begin{document}

	\title{Solving Linear Tensor Equations II: Including Parity Odd Terms in $4$-dimensions}
	
	\author{Damianos Iosifidis}
	\affiliation{Institute of Theoretical Physics, Department of Physics
		Aristotle University of Thessaloniki, 54124 Thessaloniki, Greece}
	\email{diosifid@auth.gr}
	
	\date{\today}
	\begin{abstract}
		
		In this letter, focusing in $4$-dimensions, we extend our previous results of solving linear tensor equations. In particular, we consider a $30$ parameter linear tensor equation for the unknown tensor components $N_{\alpha\mu\nu}$ in terms of the known (sources) components $B_{\alpha\mu\nu}$. Our extension now consists in including also the parity even linear terms in $N_{\alpha\mu\nu}$ (and the associated traces) formed by contracting the latter with the $4$-dimensional Levi-Civita pseudo-tensor.  Assuming a generic non-degeneracy condition and following a step by step procedure  we show how one can solve explicitly for the unknown tensor field components $N_{\alpha\mu\nu}$ and consequently derive its unique and exact solution in terms of the components $B_{\alpha\mu\nu}$.

	\end{abstract}
	
	\maketitle
	
	\allowdisplaybreaks
	
	
	
	\section{Introduction}
	Extending the results we obtained in \cite{iosifidis2021solving}, here we consider the most  general linear tensor equation in $4$-dimensions where now all possible  parity odd (i.e contractions with the Levi-Civita pseudotensor) combinations of a rank $3$ tensor\footnote{For applications of rank $3$ tensors in physics see \cite{qi2018third} (see also \cite{auffray2013geometrical}). In addition, a nice review on tensor calculus is found in \cite{landsberg2012tensors}.} are also included on top of the even ones. As we show in the proceeding discussion such a general linear tensor equation involves $30$ parameters and consists of $15$ distinct combinations of the unknown rank three tensor field $N_{\alpha\mu\nu}$. The task is to prescribe a method allowing to solve the aforementioned tensor field in terms of a given (known) one\footnote{Resorting to some decomposition scheme would not work given the complexity of the $30$ parameter tensor equation. In addition for tensors of rank $r>2$ there is no unique irreducible decomposition \cite{itin2020decomposition}.}. Such linear tensor equations appear in quadratic Metric-Affine Gravity Theories \cite{hehl1995metric,iosifidis2019metric} and their solutions give the distortion tensor\footnote{The distortion tensor is the deviation of the affine connection from the usual Levi-Civita one. See for instance \cite{schouten1954ricci,iosifidis2019metric}.} in terms of the hypermomentum tensor. Having solved for the distortion one can then readily compute spacetime torsion and non-metricity. This is but a simple application of the result which we shall present here and it can possibly be applied to other branches of physics as well. We shall start with some basic definitions and subsequently state and prove our Theorem.

	\section{Definitions}
	 Let us now give some basic definitions we are going to be using throughout. We shall consider a $4$-dimensional differentiable manifold endowed with a metric $g$ and an affine connection $\nabla$, namely $(g, \nabla, \mathcal{M})$. The signature of the metric will be denoted by $s$. In addition, in the proceeding discussion $N_{\alpha\mu\nu}$ will denote the components of an arbitrary rank-$3$ tensor field and $\varepsilon^{\alpha\beta\kappa\lambda}$ will be the components of the $4$-dimensional Levi-Civita pseudo-tensor.
	 \begin{mydef}We define the $1^{st}$, $2^{nd}$ and $3^{rd}$ contractions of $N_{\alpha\mu\nu}$ according to
	\beq
	N^{(1)}_{\mu}:=N_{\alpha\beta\mu}g^{\alpha\beta}\;\;, \;\; N^{(2)}_{\mu}:=N_{\alpha\mu\beta}g^{\alpha\beta}\;\;, \;\; N^{(3)}_{\mu}:=N_{\mu\alpha\beta}g^{\alpha\beta}
	\eeq
	respectively. 
	\end{mydef}
		\begin{mydef} Contracting $N_{\alpha\mu\nu}$ with the Levi-Civita pseudotensor we form the $3$ parity odd combinations as follows
			\beq
		M^{(1)}_{\lambda\alpha\beta}:=N_{\mu\nu\lambda}\varepsilon^{\mu\nu}_{\;\;\;\;\alpha\beta}\;\;, \;\; 	M^{(2)}_{\nu\alpha\beta}:=N_{\mu\nu\lambda}\varepsilon^{\mu\lambda}_{\;\;\;\;\alpha\beta}\;\;, \;\; M^{(3)}_{\mu\alpha\beta}:=N_{\mu\nu\lambda}\varepsilon^{\nu\lambda}_{\;\;\;\;\alpha\beta}
			\eeq
			Note that by construction each of the $M^{(i)}$'s is antisymmetric in its last pair of indices. In addition, we can construct a $4^{th}$ pseudo-trace  as
			\beq
			M^{\alpha}=N^{(4)\alpha}:=\varepsilon^{\alpha\mu\nu\lambda}N_{\mu\nu\lambda}
			\eeq
			Note also that further contractions of the $M^{(i)}$'s  do not give any new traces since $g^{\mu\nu}M^{(1)}_{\mu\nu\alpha}=-g^{\mu\nu}M^{(2)}_{\mu\nu\alpha}=g^{\mu\nu}M^{(3)}_{\mu\nu\alpha}=-M_{\alpha}$ and the rest are either identically vanishing or proportional to the latter. With the above definitions we are now in a position to state and prove the main result of this letter.
			
			\section{The Theorem}
			
		\end{mydef}
	\begin{theo}
		In a $4$-dimensional space of signature $s$, consider the $30$ parameter tensor equation
		\begin{gather}
	a_{1}N_{\alpha\mu\nu}+a_{2}N_{\nu\alpha\mu}+a_{3}N_{\mu\nu\alpha}+a_{4}N_{\alpha\nu\mu}+a_{5}N_{\nu\mu\alpha}+a_{6}N_{\mu\alpha\nu}+\sum_{i=1}^{3}\Big( a_{7 i}N^{(i)}_{\mu}g_{\alpha\nu}+a_{8 i}N^{(i)}_{\nu}g_{\alpha\mu}+a_{9 i}N^{(i)}_{\alpha}g_{\mu\nu} \Big) \nonumber \\
	+b_{11}M^{(1)}_{\alpha\mu\nu}+b_{12}M^{(1)}_{\nu\alpha\mu}+b_{13}M^{(1)}_{\mu\nu\alpha}+b_{21}M^{(2)}_{\alpha\mu\nu}+b_{22}M^{(2)}_{\nu\alpha\mu}+b_{23}M^{(2)}_{\mu\nu\alpha}+b_{31}M^{(3)}_{\alpha\mu\nu}+b_{32}M^{(3)}_{\nu\alpha\mu}+b_{33}M^{(3)}_{\mu\nu\alpha} \nonumber \\
+\varepsilon_{\rho\alpha\mu\nu}\Big( b_{1}N^{(1)\rho}+ b_{2}N^{(2)\rho}+ b_{3}N^{(3)\rho}\Big)	+c_{1}M_{\mu}g_{\alpha\nu}+c_{2}M_{\nu}g_{\alpha\mu}+c_{3}M_{\alpha}g_{\mu\nu}
	=B_{\alpha\mu\nu} \label{eq1}
	\end{gather}
		where  $a_{i}$, $a_{ji}$ $i=1,2,...,6$, $j=7,8,9$ are scalars, $b_{kl}$, $c_{m}$ are pseudo-scalars,  $B_{\alpha\mu\nu}$   is a given (known) tensor and $N_{\alpha\mu\nu}$ are the components of the unknown tensor\footnote{Of course the result holds true even when $N_{\alpha\mu\nu}$ are the components of a tensor density instead or even of a connection given that $B_{\alpha\mu\nu}$ are also  of the same kind. } $N$. Define the matrices
		\setcounter{MaxMatrixCols}{20}
		\begin{equation}
		A := 
		\begin{pmatrix}
		\alpha_{11} & \alpha_{12} & \alpha_{13} & \alpha_{14} & \alpha_{15} & \alpha_{16} & \alpha_{17} & \alpha_{18} & \alpha_{19} & \alpha_{1,10}&\alpha_{1,11} & \alpha_{1,12} & \alpha_{1,13} & \alpha_{1,14} & \alpha_{1,15}  \\
	\alpha_{21} & \alpha_{22} & \alpha_{23} & \alpha_{24} & \alpha_{25} & \alpha_{26} & \alpha_{27} & \alpha_{28} & \alpha_{29} & \alpha_{2,10} & \alpha_{2,11} & \alpha_{1,12} & \alpha_{2,13} & \alpha_{2,14} & \alpha_{2,15} \\
	\alpha_{31} & \alpha_{32} & \alpha_{33} & \alpha_{34} & \alpha_{35} & \alpha_{36} & \alpha_{37} & \alpha_{38} & \alpha_{39} & \alpha_{3,10} & \alpha_{3,11} & \alpha_{1,12} & \alpha_{1,13} & \alpha_{1,14} & \alpha_{1,15}  \\
	\alpha_{41} & \alpha_{42} & \alpha_{43} & \alpha_{44} & \alpha_{45} & \alpha_{46} & \alpha_{47} & \alpha_{48} & \alpha_{49} & \alpha_{4,10} & \alpha_{4,11} &  \alpha_{4,12} & \alpha_{4,13} & \alpha_{4,14} & \alpha_{4,15} \\
	\alpha_{51} & \alpha_{52} & \alpha_{53} & \alpha_{54} & \alpha_{55} & \alpha_{56} & \alpha_{57} & \alpha_{58} & \alpha_{59} & \alpha_{5,10} & \alpha_{5,11} & \alpha_{5,12} & \alpha_{5,13} & \alpha_{5,14} & \alpha_{5,15} \\
	\alpha_{61} & \alpha_{62} & \alpha_{63} & \alpha_{64} & \alpha_{65} & \alpha_{66} & \alpha_{67} & \alpha_{68} & \alpha_{69} & \alpha_{6,10} & \alpha_{6,11} & \alpha_{6,12} & \alpha_{6,13} & \alpha_{6,14} & \alpha_{6,15}  \\
		\alpha_{71} & \alpha_{72} & \alpha_{73} & \alpha_{74} & \alpha_{75} & \alpha_{76} & \alpha_{77} & \alpha_{78} & \alpha_{79} & \alpha_{7,10} & \alpha_{7,11} & \alpha_{7,12} & \alpha_{7,13} & \alpha_{7,14} & \alpha_{7,15}  \\
	\alpha_{81} & \alpha_{82} & \alpha_{83} & \alpha_{84} & \alpha_{85} & \alpha_{86} & \alpha_{87} & \alpha_{88} & \alpha_{89} & \alpha_{8,10} & \alpha_{8,11} & \alpha_{8,12} & \alpha_{8,13} & \alpha_{8,14} & \alpha_{8,15} \\
	\alpha_{91} & \alpha_{92} & \alpha_{93} & \alpha_{94} & \alpha_{95} & \alpha_{96} & \alpha_{97} & \alpha_{98} & \alpha_{99} & \alpha_{9,10} & \alpha_{9,11} &  \alpha_{9,12} &\alpha_{9,13} & \alpha_{9,14} & \alpha_{9,15}  \\
	\alpha_{10,1} & \alpha_{10,2} & \alpha_{10,3} & \alpha_{10,4} & \alpha_{10,5} & \alpha_{10,6} & \alpha_{10,7} & \alpha_{10,8} & \alpha_{10,9} & \alpha_{10,10} & \alpha_{10,11} & \alpha_{10,12} & \alpha_{10,13} & \alpha_{10,14} & \alpha_{10,15} \\
	\alpha_{11,1} & \alpha_{11,2} & \alpha_{11,3} & \alpha_{11,4} & \alpha_{11,5} & \alpha_{11,6} & \alpha_{11,7} & \alpha_{11,8} & \alpha_{11,9} & \alpha_{11,10} & \alpha_{11,11} & \alpha_{11,12} & \alpha_{11,13} & \alpha_{11,14} & \alpha_{11,15} \\
	\alpha_{12,1} & \alpha_{12,2} & \alpha_{12,3} & \alpha_{12,4} & \alpha_{12,5} & \alpha_{12,6} & \alpha_{12,7} & \alpha_{12,8} & \alpha_{12,9} & \alpha_{12,10} & \alpha_{12,11} & \alpha_{12,12} & \alpha_{12,13} & \alpha_{12,14} & \alpha_{12,15}  \\
	\alpha_{13,1} & \alpha_{13,2} & \alpha_{13,3} & \alpha_{14,4} & \alpha_{14,5} & \alpha_{14,6} & \alpha_{14,7} & \alpha_{14,8} & \alpha_{14,9} & \alpha_{1,10} & \alpha_{1,11} & \alpha_{1,12} & \alpha_{1,13} & \alpha_{1,14} & \alpha_{1,15}  \\
		\alpha_{11} & \alpha_{12} & \alpha_{13} & \alpha_{14} & \alpha_{15} & \alpha_{16} & \alpha_{17} & \alpha_{18} & \alpha_{19} & \alpha_{1,10} & \alpha_{14,11} & \alpha_{14,12} & \alpha_{14,13} & \alpha_{14,14} & \alpha_{14,15}  \\
	\alpha_{15,1} & \alpha_{15,2} & \alpha_{15,3} & \alpha_{15,4} & \alpha_{15,5} & \alpha_{15,6} & \alpha_{15,7} & \alpha_{15,8} & \alpha_{15,9} & \alpha_{15,10} & \alpha_{15,11} &  \alpha_{15,12} &\alpha_{15,13} & \alpha_{15,14} & \alpha_{15,15}  \\
		\end{pmatrix} \label{A}
		\end{equation}
		and 
		\begin{equation}
	\Gamma := 
	\begin{pmatrix}
	\gamma_{11} & \gamma_{12} & \gamma_{13} & \gamma_{14}  \\
	\gamma_{21} & \gamma_{22} & \gamma_{23} & \gamma_{24}  \\
	\gamma_{31} & \gamma_{32} & \gamma_{33} & \gamma_{34}  \\
	\gamma_{41} & \gamma_{42} & \gamma_{43} & \gamma_{44}  \\
	\end{pmatrix} \label{G}
	\end{equation}		
		where  the $\alpha_{ij}$ and  $\gamma_{ij}$'s are some linear combinations of the original $a_{i}$, $a_{ji}$, $b_{j}$ and $c_{k}$'s (see appendix). Then given that both of these matrices are non-singular, that is if	
		\beq
		det(A) \neq 0 \;\; and \;\; det(\Gamma) \neq 0
		\eeq
		holds true, then the general and unique solution of ($\ref{eq1}$) reads
	\begin{gather}
	N_{\alpha\mu\nu}=\tilde{\alpha}_{11}\hat{B}_{\alpha\mu\nu}+\tilde{\alpha}_{12}\hat{B}_{\nu\alpha\mu}+\tilde{\alpha}_{13} \hat{B}_{\mu\nu\alpha}+\tilde{\alpha}_{14} \hat{B}_{\alpha\nu\mu}+\tilde{\alpha}_{15} \hat{B}_{\nu\mu\alpha} +\tilde{\alpha}_{16} \hat{B}_{\mu\alpha\nu}+\tilde{\alpha}_{17}\breve{B}_{\alpha\mu\nu}+\tilde{\alpha}_{18} \breve{B}_{\mu\nu\alpha} + \tilde{\alpha}_{19}\breve{B}_{\nu\alpha\mu} \nonumber \\
	+\tilde{\alpha}_{1,10}\bar{B}_{\alpha\mu\nu}+\tilde{\alpha}_{1,11}\bar{B}_{\mu\nu\alpha}+\tilde{\alpha}_{1,12}\bar{B}_{\nu\alpha\mu}+\tilde{\alpha}_{1,13}\mathring{B}_{\alpha\mu\nu}+\tilde{\alpha}_{1,14}\mathring{B}_{\mu\nu\alpha}+\tilde{\alpha}_{1,15}\mathring{B}_{\nu\alpha\mu}
	\end{gather}
		where the $\tilde{\alpha}_{1i}$ $  's$ are the first row elements of the inverse matrix $A^{-1}$ and 
			\beq
		\hat{B}_{\alpha\mu\nu} =B_{\alpha\mu\nu} -\sum_{i=1}^{4}\sum_{j=1}^{4}\Big( a_{7i}\tilde{\gamma}_{ij}B^{(j)}_{\mu}g_{\alpha\nu}+a_{8i}\tilde{\gamma}_{ij}B^{(j)}_{\nu}g_{\alpha\mu}+ a_{9i}\tilde{\gamma}_{ij}B^{(j)}_{\alpha}g_{\mu\nu} \Big)-\varepsilon^{\rho}_{\;\;\alpha\mu\nu}\sum_{i=1}^{3}\sum_{j=1}^{4}b_{i}\tilde{\gamma}_{ij}B^{(j)}_{\rho}
		\eeq
		\beq
		\breve{B}_{\alpha\mu\nu}=	\varepsilon^{\beta\gamma}_{\;\;\;\;\alpha\mu}\hat{B}_{\beta\gamma\nu}-2(-1)^{s}\sum_{j=1}^{4}\Big[ (b_{21}+b_{23}+b_{31}+b_{33})\tilde{\gamma}_{1j}+(b_{11}+b_{13}-b_{31}-b_{33})\tilde{\gamma}_{2j}-(b_{11}+b_{13}+b_{21}+b_{23})\tilde{\gamma}_{3j}\Big]B^{(j)}_{[\alpha}g_{\mu]\nu}
		\eeq
		\beq
		\bar{B}_{\alpha\mu\nu}:=	\varepsilon^{\beta\gamma}_{\;\;\;\;\alpha\nu}\hat{B}_{\beta\mu\gamma}-2(-1)^{s}\sum_{j=1}^{4}\Big[ (b_{21}+b_{22}+b_{31}+b_{32})\tilde{\gamma}_{1j}+(b_{11}+b_{12}-b_{31}-b_{32})\tilde{\gamma}_{2j}-(b_{11}+b_{12}+b_{21}+b_{22})\tilde{\gamma}_{3j}\Big]B^{(j)}_{[\alpha}g_{\mu]\nu}
		\eeq
		\beq
		\mathring{B}_{\alpha\mu\nu}:=	\varepsilon^{\beta\gamma}_{\;\;\;\;\mu\nu}	\mathring{B}_{\alpha\beta\gamma}-2(-1)^{s+1}\sum_{j=1}^{4}\Big[ (b_{22}+b_{23}+b_{32}+b_{33})\tilde{\gamma}_{1j}+(b_{12}+b_{13}-b_{32}-b_{33})\tilde{\gamma}_{2j}-(b_{22}+b_{23}+b_{12}+b_{13})\tilde{\gamma}_{3j}\Big]B^{(j)}_{[\mu}g_{\nu]\alpha}
		\eeq
		$\tilde{\gamma}_{ij}'s$ being the elements of the inverse matrix $\Gamma^{-1}$.
	\end{theo}

	\begin{proof}
		  The first step consists in removing the traces (and pseudo-trace) of $N$ and expressing them in terms of corresponding traces of the known tensor $B$. To this end we perform $4$ distinct operations on $(\ref{eq1})$ that is we contract the latter respectively with $g^{\alpha\mu}$, $g^{\alpha\nu}$, $g^{\mu\nu}$ and $\varepsilon^{\lambda\alpha\mu\nu}$ which, after some renaming of the indices, give us the system
		  \begin{gather}
		  \gamma_{11}N^{(1)}_{\mu}+\gamma_{12}N^{(2)}_{\mu}+\gamma_{13}N^{(3)}_{\mu}+\gamma_{14}N^{(4)}_{\mu}=B^{(1)}_{\mu}  \\
		  \gamma_{21}N^{(1)}_{\mu}+\gamma_{22}N^{(2)}_{\mu}+\gamma_{23}N^{(3)}_{\mu}+\gamma_{24}N^{(4)}_{\mu}=B^{(2)}_{\mu}  \\
		  \gamma_{31}N^{(1)}_{\mu}+\gamma_{32}N^{(2)}_{\mu}+\gamma_{33}N^{(3)}_{\mu}+\gamma_{34}N^{(4)}_{\mu}=B^{(3)}_{\mu}  \\
		  \gamma_{41}N^{(1)}_{\mu}+\gamma_{42}N^{(2)}_{\mu}+\gamma_{43}N^{(3)}_{\mu}+\gamma_{44}N^{(4)}_{\mu}=B^{(4)}_{\mu} 
		  \end{gather}	
		where $\gamma_{ij}$ are linear combinations of the initial $27$ parameters whose exact form we give in the appendix. We have also used the notation $M_{\mu}=N^{(4)}_{\mu}$. The above is a system of $4$ equations with $4$ unknowns which we may express in matrix form as
		\beq
		\Gamma X=Y \label{tr}
		\eeq
		where we have defined the columns $X:=(N^{(1)}_{\mu},N^{(2)}_{\mu},N^{(3)}_{\mu},N^{(4)}_{\mu})^{T}$ and $Y:=(B^{(1)}_{\mu},B^{(2)}_{\mu},B^{(3)}_{\mu},B^{(4)}_{\mu})^{T}$ along with the matrix
			\begin{equation}
		\Gamma := 
		\begin{pmatrix}
		\gamma_{11} & \gamma_{12} & \gamma_{13} & \gamma_{14}  \\
		\gamma_{21} & \gamma_{22} & \gamma_{23} & \gamma_{24}  \\
		\gamma_{31} & \gamma_{32} & \gamma_{33} & \gamma_{34}  \\
		\gamma_{41} & \gamma_{42} & \gamma_{43} & \gamma_{44}  \\
		\end{pmatrix} \label{G}
		\end{equation}
		Now, given that the matrix $\Gamma$ is not degenerate, that is if
		\beq
		det(\Gamma)\neq 0 \label{C}
		\eeq
		then, its inverse $\Gamma^{-1}$ exists. Then, multiplying $(\ref{tr})$ with the latter from the left we get
		\beq
		X=\Gamma^{-1}Y
		\eeq
		which relates the unknown traces $N^{(i)}_{\mu}$ to the known ones $(B^{(i)}_{\mu})$. More specifically, the component form of the above matrix equation gives us the relations
		\beq
	N^{(i)}_{\mu}=\sum_{j=1}^{4}\tilde{\gamma}_{ij}B^{(j)}_{\mu}	\label{NB}
		\eeq
	where $\tilde{\gamma}_{ij}$ are the elements of the inverse matrix $\Gamma^{-1}$. As a result we have fully eliminated the N-traces in terms of those of $B$. We may then substitute the last equation back in ($\ref{eq1}$) to arrive at
		\begin{gather}
	a_{1}N_{\alpha\mu\nu}+a_{2}N_{\nu\alpha\mu}+a_{3}N_{\mu\nu\alpha}+a_{4}N_{\alpha\nu\mu}+a_{5}N_{\nu\mu\alpha}+a_{6}N_{\mu\alpha\nu} \nonumber \\
	+b_{11}M^{(1)}_{\alpha\mu\nu}+b_{12}M^{(1)}_{\nu\alpha\mu}+b_{13}M^{(1)}_{\mu\nu\alpha}+b_{21}M^{(2)}_{\alpha\mu\nu}+b_{22}M^{(2)}_{\nu\alpha\mu}+b_{23}M^{(2)}_{\mu\nu\alpha}+b_{31}M^{(3)}_{\alpha\mu\nu}+b_{32}M^{(3)}_{\nu\alpha\mu}+b_{33}M^{(3)}_{\mu\nu\alpha} 
	=\hat{B}_{\alpha\mu\nu} \label{eq2}
	\end{gather}
	where 
	\beq
	\hat{B}_{\alpha\mu\nu} =B_{\alpha\mu\nu} -\sum_{i=1}^{4}\sum_{j=1}^{4}\Big( a_{7i}\tilde{\gamma}_{ij}B^{(j)}_{\mu}g_{\alpha\nu}+a_{8i}\tilde{\gamma}_{ij}B^{(j)}_{\nu}g_{\alpha\mu}+ a_{9i}\tilde{\gamma}_{ij}B^{(j)}_{\alpha}g_{\mu\nu} \Big)-\varepsilon^{\rho}_{\;\;\alpha\mu\nu}\sum_{i=1}^{3}\sum_{j=1}^{4}b_{i}\tilde{\gamma}_{ij}B^{(j)}_{\rho}
\eeq
	and we have renamed $c_{1}=a_{74}$, $c_{2}=a_{84}$ and $c_{3}=a_{94}$ in order to obtain a more compact form. We have now  solved only for the $4$ out of the $19$ unknown combinations as they appear in ($\ref{eq1}$). There remain $15$ more thus we need $15$ more equations in order to fully solve for $N$. To this end, we now consider the  $5$  possible independent  permutations of $(\ref{eq2})$ which read (including also $(\ref{eq2})$ itself and using shorthand sum notation)\footnote{See  also \cite{iosifidis2021solving}.}

		\begin{gather}
		a_{1}N_{\alpha\mu\nu}+a_{2}N_{\nu\alpha\mu}+a_{3}N_{\mu\nu\alpha}+a_{4}N_{\alpha\nu\mu}+a_{5}N_{\nu\mu\alpha}+a_{6}N_{\mu\alpha\nu}  +\sum_{i=1}^{3}\Big( b_{i1}M^{(i)}_{\alpha\mu\nu}+ b_{i2}M^{(i)}_{\nu\alpha\mu}+ b_{i3}M^{(i)}_{\mu\nu\alpha}\Big)	=\hat{B}_{\alpha\mu\nu} \label{1} \\ 
		a_{1}N_{\nu\alpha\mu}+a_{2}N_{\mu\nu\alpha}+a_{3}N_{\alpha\mu\nu}+a_{4}N_{\mu\alpha\nu}+a_{5}N_{\alpha\nu\mu}+a_{6}N_{\nu\mu\alpha}+\sum_{i=1}^{3}\Big( b_{i1}M^{(i)}_{\nu\alpha\mu}+ b_{i2}M^{(i)}_{\mu\nu\alpha}+ b_{i3}M^{(i)}_{\alpha\mu\nu}\Big)	=\hat{B}_{\nu\alpha\mu} \\
		a_{1}N_{\mu\nu\alpha}+a_{2}N_{\alpha\mu\nu}+a_{3}N_{\nu\alpha\mu}+a_{4}N_{\nu\mu\alpha}+a_{5}N_{\mu\alpha\nu}+a_{6}N_{\alpha\nu\mu}+\sum_{i=1}^{3}\Big( b_{i1}M^{(i)}_{\mu\nu\alpha}+ b_{i2}M^{(i)}_{\alpha\mu\nu}+ b_{i3}M^{(i)}_{\nu\alpha\mu}\Big)	=\hat{B}_{\mu\nu\alpha} \\
		a_{1}N_{\alpha\nu\mu}+a_{2}N_{\mu\alpha\nu}+a_{3}N_{\nu\mu\alpha}+a_{4}N_{\alpha\mu\nu}+a_{5}N_{\mu\nu\alpha}+a_{6}N_{\nu\alpha\mu} +\sum_{i=1}^{3}\Big( b_{i1}M^{(i)}_{\alpha\nu\mu}+ b_{i2}M^{(i)}_{\mu\alpha\nu}+ b_{i3}M^{(i)}_{\nu\mu\alpha}\Big)	=\hat{B}_{\alpha\nu\mu} \\
		a_{1}N_{\nu\mu\alpha}+a_{2}N_{\alpha\nu\mu}+a_{3}N_{\mu\alpha\nu}+a_{4}N_{\nu\alpha\mu}+a_{5}N_{\alpha\mu\nu}+a_{6}N_{\mu\nu\alpha}+\sum_{i=1}^{3}\Big( b_{i1}M^{(i)}_{\nu\mu\alpha}+ b_{i2}M^{(i)}_{\alpha\nu\mu}+ b_{i3}M^{(i)}_{\mu\alpha\nu}\Big)	=\hat{B}_{\nu\mu\alpha}\\
		a_{1}N_{\mu\alpha\nu}+a_{2}N_{\nu\mu\alpha}+a_{3}N_{\alpha\nu\mu}+a_{4}N_{\mu\nu\alpha}+a_{5}N_{\nu\alpha\mu}+a_{6}N_{\alpha\mu\nu}+\sum_{i=1}^{3}\Big( b_{i1}M^{(i)}_{\mu\alpha\nu}+ b_{i2}M^{(i)}_{\nu\mu\alpha}+ b_{i3}M^{(i)}_{\alpha\nu\mu}\Big)	=\hat{B}_{\mu\alpha\nu} \label{6}
		\end{gather}
		We have therefore gathered $6$ equations and we need $9$ more to go. Continuing we now contract ($\ref{eq2}$) with $\varepsilon^{\alpha\mu}_{\;\;\;\;\beta\gamma}$ and upon using the identity $\varepsilon_{\alpha\mu\beta\gamma}\varepsilon^{\kappa\lambda\rho\sigma}=(-1)^{s}4! \delta_{[\alpha}^{\kappa}\delta_{\mu}^{\lambda}\delta_{\beta}^{\rho}\delta_{\gamma]}^{\sigma}$ and its contractions, after some long calculations we finally arrive at
		\begin{gather}
		(a_{1}-a_{6})M^{(1)}_{\nu\beta\gamma}+(a_{4}-a_{3})M^{(2)}_{\nu\beta\gamma}+(a_{2}-a_{5})M^{(3)}_{\nu\beta\gamma}+ 2 (-1)^{s+1} (b_{21}+b_{23}+b_{31}+b_{33})N^{(1)}_{[\beta}g_{\gamma]\nu} \nonumber \\ +2 (-1)^{s+1}(b_{11}+b_{13}-b_{31}-b_{33})N^{(2)}_{[\beta}g_{\gamma]\nu}-2 (-1)^{s+1}(b_{11}+b_{13}+b_{21}+b_{23})N^{(3)}_{[\beta} g_{\gamma]\nu} \nonumber \\
		+2(-1)^{s}\Big[ -(b_{11}+b_{13})+2 b_{12}\Big] N_{[\beta\gamma]\nu}+2(-1)^{s}\Big[ -(b_{21}+b_{23})+2 b_{22}\Big] N_{[\beta|\nu|\gamma]}+2(-1)^{s}\Big[ -(b_{31}+b_{33})+2 b_{32}\Big] N_{\nu[\beta\gamma]}\nonumber \\=\varepsilon^{\alpha\mu}_{\;\;\;\;\beta\gamma}\hat{B}_{\alpha\mu\nu}
		\end{gather}
		Now renaming the indices $\beta\rightarrow \alpha$, $\gamma \rightarrow \mu$ and using ($\ref{NB}$) to remove the $N$-traces, after some rearranging we end up with
			\begin{gather}
(-1)^{s}( 2 b_{12}-b_{11}-b_{13})N_{\alpha\mu\nu}+(-1)^{s}(2 b_{32}-b_{31}-b_{33})N_{\nu\alpha\mu}-(-1)^{s}(2 b_{22}-b_{21}-b_{23})N_{\mu\nu\alpha}\nonumber \\
+(-1)^{s}(2 b_{22}-b_{21}-b_{23})N_{\alpha\nu\mu}-(-1)^{s}(2 b_{32}-b_{31}-b_{33})N_{\nu\mu\alpha}-(-1)^{s}(2 b_{12}-b_{11}-b_{13})N_{\mu\alpha\nu} \nonumber \\
+(a_{1}-a_{6})M^{(1)}_{\nu\alpha\mu}+(a_{4}-a_{3})M^{(2)}_{\nu\alpha\mu}+(a_{2}-a_{5})M^{(3)}_{\nu\alpha\mu}=	\breve{B}_{\alpha\mu\nu} \label{e1}
		\end{gather}
		where we have set
		\beq
	\breve{B}_{\alpha\mu\nu}=	\varepsilon^{\beta\gamma}_{\;\;\;\;\alpha\mu}\hat{B}_{\beta\gamma\nu}-2(-1)^{s}\sum_{j=1}^{4}\Big[ (b_{21}+b_{23}+b_{31}+b_{33})\tilde{\gamma}_{1j}+(b_{11}+b_{13}-b_{31}-b_{33})\tilde{\gamma}_{2j}-(b_{11}+b_{13}+b_{21}+b_{23})\tilde{\gamma}_{3j}\Big]B^{(j)}_{[\alpha}g_{\mu]\nu}
		\eeq
		Next we consider the index permutation $\alpha \rightarrow \mu \rightarrow \nu \rightarrow \alpha$ in ($\ref{e1}$) firstly one time and then two successive ones to obtain two more equations
		\begin{gather}
		(-1)^{s}( 2 b_{12}-b_{11}-b_{13})N_{\mu\nu\alpha}+(-1)^{s}(2 b_{32}-b_{31}-b_{33})N_{\alpha\mu\nu}-(-1)^{s}(2 b_{22}-b_{21}-b_{23})N_{\nu\alpha\mu}\nonumber \\
		+(-1)^{s}(2 b_{22}-b_{21}-b_{23})N_{\mu\alpha\nu}-(-1)^{s}(2 b_{32}-b_{31}-b_{33})N_{\alpha\nu\mu}-(-1)^{s}(2 b_{12}-b_{11}-b_{13})N_{\nu\mu\alpha} \nonumber \\
		+(a_{1}-a_{6})M^{(1)}_{\alpha\mu\nu}+(a_{4}-a_{3})M^{(2)}_{\alpha\mu\nu}+(a_{2}-a_{5})M^{(3)}_{\alpha\mu\nu}=	\breve{B}_{\mu\nu\alpha} \label{e2}
		\end{gather}
		\begin{gather}
		(-1)^{s}( 2 b_{12}-b_{11}-b_{13})N_{\nu\alpha\mu}+(-1)^{s}(2 b_{32}-b_{31}-b_{33})N_{\mu\nu\alpha}-(-1)^{s}(2 b_{22}-b_{21}-b_{23})N_{\alpha\mu\nu}\nonumber \\
		+(-1)^{s}(2 b_{22}-b_{21}-b_{23})N_{\nu\mu\alpha}-(-1)^{s}(2 b_{32}-b_{31}-b_{33})N_{\mu\alpha\nu}-(-1)^{s}(2 b_{12}-b_{11}-b_{13})N_{\alpha\nu\mu} \nonumber \\
		+(a_{1}-a_{6})M^{(1)}_{\mu\nu\alpha}+(a_{4}-a_{3})M^{(2)}_{\mu\nu\alpha}+(a_{2}-a_{5})M^{(3)}_{\mu\nu\alpha}=	\breve{B}_{\nu\alpha\mu} \label{e3}
		\end{gather}
		In the same manner, we now contract $(\ref{eq2})$ one time with $\varepsilon^{\alpha\nu}_{\;\;\;\;\beta\gamma}$ and  another with $\varepsilon^{\mu\nu}_{\;\;\;\;\beta\gamma}$ and again performing one and two successive permutations of the indices for each case respectively, we gather six more equations $(3+3)$
			\begin{gather}
(-1)^{s}(b_{21}+b_{22}-2 b_{23})N_{\alpha\mu\nu}-(-1)^{s}(b_{11}+b_{12}-2 b_{13})N_{\nu\alpha\mu}-(-1)^{s}(b_{31}+b_{32}-2 b_{33})N_{\mu\nu\alpha} \nonumber \\
+(-1)^{s}(b_{11}+b_{12}-2 b_{13})N_{\alpha\nu\mu}-(-1)^{s}(b_{21}+b_{22}-2 b_{23})N_{\nu\mu\alpha}+(-1)^{s}(b_{31}+b_{32}-2 b_{33})N_{\mu\alpha\nu} \nonumber \\
+(a_{4}-a_{2})M^{(1)}_{\mu\alpha\nu}+(a_{1}-a_{5})M^{(2)}_{\mu\alpha\nu}+(a_{6}-a_{3})M^{(3)}_{\mu\alpha\nu}=\bar{B}_{\alpha\mu\nu} \label{ee1}
		\end{gather}
		 	\begin{gather}
		 (-1)^{s}(b_{21}+b_{22}-2 b_{23})N_{\mu\nu\alpha}-(-1)^{s}(b_{11}+b_{12}-2 b_{13})N_{\alpha\mu\nu}-(-1)^{s}(b_{31}+b_{32}-2 b_{33})N_{\nu\alpha\mu} \nonumber \\
		 +(-1)^{s}(b_{11}+b_{12}-2 b_{13})N_{\mu\alpha\nu}-(-1)^{s}(b_{21}+b_{22}-2 b_{23})N_{\alpha\nu\mu}+(-1)^{s}(b_{31}+b_{32}-2 b_{33})N_{\nu\mu\alpha} \nonumber \\
		 +(a_{4}-a_{2})M^{(1)}_{\nu\mu\alpha}+(a_{1}-a_{5})M^{(2)}_{\nu\mu\alpha}+(a_{6}-a_{3})M^{(3)}_{\nu\mu\alpha}=\bar{B}_{\mu\nu\alpha}
		 \end{gather}
		 	\begin{gather}
		 (-1)^{s}(b_{21}+b_{22}-2 b_{23})N_{\nu\alpha\mu}-(-1)^{s}(b_{11}+b_{12}-2 b_{13})N_{\mu\nu\alpha}-(-1)^{s}(b_{31}+b_{32}-2 b_{33})N_{\alpha\mu\nu} \nonumber \\
		 +(-1)^{s}(b_{11}+b_{12}-2 b_{13})N_{\nu\mu\alpha}-(-1)^{s}(b_{21}+b_{22}-2 b_{23})N_{\mu\alpha\nu}+(-1)^{s}(b_{31}+b_{32}-2 b_{33})N_{\alpha\nu\mu} \nonumber \\
		 +(a_{4}-a_{2})M^{(1)}_{\alpha\nu\mu}+(a_{1}-a_{5})M^{(2)}_{\alpha\nu\mu}+(a_{6}-a_{3})M^{(3)}_{\alpha\nu\mu}=\bar{B}_{\nu\alpha\mu}
		 \end{gather}
		
			\begin{gather}
		(-1)^{s}(2 b_{31}-b_{32}- b_{33})N_{\alpha\mu\nu}-(-1)^{s}(2 b_{21}-b_{22}- b_{23})N_{\nu\alpha\mu}+(-1)^{s}(2 b_{11}-b_{12}- b_{13})N_{\mu\nu\alpha} \nonumber \\
		-(-1)^{s}(2 b_{31}-b_{32}- b_{33})N_{\alpha\nu\mu}-(-1)^{s}(2 b_{11}-b_{12}- b_{13})N_{\nu\mu\alpha}+(-1)^{s}(2 b_{21}-b_{22}- b_{23})N_{\mu\alpha\nu} \nonumber \\
		+(a_{3}-a_{5})M^{(1)}_{\alpha\mu\nu}+(a_{6}-a_{2})M^{(2)}_{\alpha\mu\nu}+(a_{1}-a_{4})M^{(3)}_{\alpha\mu\nu}=\mathring{B}_{\alpha\mu\nu}
		\end{gather}
		\begin{gather}
		(-1)^{s}(2 b_{31}-b_{32}- b_{33})N_{\mu\nu\alpha}-(-1)^{s}(2 b_{21}-b_{22}- b_{23})N_{\alpha\mu\nu}+(-1)^{s}(2 b_{11}-b_{12}- b_{13})N_{\nu\alpha\mu} \nonumber \\
		-(-1)^{s}(2 b_{31}-b_{32}- b_{33})N_{\mu\alpha\nu}-(-1)^{s}(2 b_{11}-b_{12}- b_{13})N_{\alpha\nu\mu}+(-1)^{s}(2 b_{21}-b_{22}- b_{23})N_{\nu\mu\alpha} \nonumber \\
		+(a_{3}-a_{5})M^{(1)}_{\mu\nu\alpha}+(a_{6}-a_{2})M^{(2)}_{\mu\nu\alpha}+(a_{1}-a_{4})M^{(3)}_{\mu\nu\alpha}=\mathring{B}_{\mu\nu\alpha}
		\end{gather}
		\begin{gather}
		(-1)^{s}(2 b_{31}-b_{32}- b_{33})N_{\nu\alpha\mu}-(-1)^{s}(2 b_{21}-b_{22}- b_{23})N_{\mu\nu\alpha}+(-1)^{s}(2 b_{11}-b_{12}- b_{13})N_{\alpha\mu\nu} \nonumber \\
		-(-1)^{s}(2 b_{31}-b_{32}- b_{33})N_{\nu\mu\alpha}-(-1)^{s}(2 b_{11}-b_{12}- b_{13})N_{\mu\alpha\nu}+(-1)^{s}(2 b_{21}-b_{22}- b_{23})N_{\alpha\nu\mu} \nonumber \\
		+(a_{3}-a_{5})M^{(1)}_{\nu\alpha\mu}+(a_{6}-a_{2})M^{(2)}_{\nu\alpha\mu}+(a_{1}-a_{4})M^{(3)}_{\nu\alpha\mu}=\mathring{B}_{\nu\alpha\mu} \label{ee6}
		\end{gather}
			where we have set
		\beq
		\bar{B}_{\alpha\mu\nu}:=	\varepsilon^{\beta\gamma}_{\;\;\;\;\alpha\nu}\hat{B}_{\beta\mu\gamma}-2(-1)^{s}\sum_{j=1}^{4}\Big[ (b_{21}+b_{22}+b_{31}+b_{32})\tilde{\gamma}_{1j}+(b_{11}+b_{12}-b_{31}-b_{32})\tilde{\gamma}_{2j}-(b_{11}+b_{12}+b_{21}+b_{22})\tilde{\gamma}_{3j}\Big]B^{(j)}_{[\alpha}g_{\mu]\nu}
		\eeq
	and
			\beq
		\mathring{B}_{\alpha\mu\nu}:=	\varepsilon^{\beta\gamma}_{\;\;\;\;\mu\nu}	\mathring{B}_{\alpha\beta\gamma}-2(-1)^{s+1}\sum_{j=1}^{4}\Big[ (b_{22}+b_{23}+b_{32}+b_{33})\tilde{\gamma}_{1j}+(b_{12}+b_{13}-b_{32}-b_{33})\tilde{\gamma}_{2j}-(b_{22}+b_{23}+b_{12}+b_{13})\tilde{\gamma}_{3j}\Big]B^{(j)}_{[\mu}g_{\nu]\alpha}
		\eeq
		We now place equations $(\ref{1})-(\ref{6})$, $(\ref{e1})-(\ref{e3})$ and $(\ref{ee1})-(\ref{ee6})$ in that exact order and express the system of the above $15$ equations in matrix form as
		\beq
		A \mathcal{N}=\mathcal{B}
		\eeq
		where $A$
		is the $15\times 15$ matrix of coefficients (see appendix) and we have also defined the columns 
		\beq \mathcal{N}=\Big(N_{\alpha\mu\nu},N_{\nu\alpha\mu},N_{\mu\nu\alpha},N_{\alpha\nu\mu},N_{\nu\mu\alpha},N_{\mu\alpha\nu},M^{(1)}_{\alpha\mu\nu},M^{(1)}_{\nu\alpha\mu},M^{(1)}_{\mu\nu\alpha},M^{(2)}_{\alpha\mu\nu},M^{(2)}_{\nu\alpha\mu},M^{(2)}_{\mu\nu\alpha},M^{(3)}_{\alpha\mu\nu},M^{(3)}_{\nu\alpha\mu},M^{(3)}_{\mu\nu\alpha}\Big)^{T}
		\eeq
		as well as
			\beq \mathcal{B}=\Big(\hat{B}_{\alpha\mu\nu},\hat{B}_{\nu\alpha\mu},\hat{B}_{\mu\nu\alpha},\hat{B}_{\alpha\nu\mu},\hat{B}_{\nu\mu\alpha},\hat{B}_{\mu\alpha\nu},\breve{B}_{\alpha\mu\nu},\breve{B}_{\mu\nu\alpha},\breve{B}_{\nu\alpha\mu},\bar{B}_{\alpha\mu\nu},\bar{B}_{\mu\nu\alpha},\bar{B}_{\nu\alpha\mu},\mathring{B}_{\alpha\mu\nu},\mathring{B}_{\mu\nu\alpha},\mathring{B}_{\nu\alpha\mu}\Big)^{T}
		\eeq
		Then, given that the matrix $A$ is non-singular, namely
		\beq
		 det(A)\neq 0 \label{A}
		 \eeq
		 we can formally multiply the above matrix equation with $A^{-1}$ from the left, to get
		\beq
			 \mathcal{N}=A^{-1}\mathcal{B}
		\eeq
		and finally equating the first elements of the latter column equation, we arrive at the stated result 
		\begin{gather}
		N_{\alpha\mu\nu}=\tilde{\alpha}_{11}\hat{B}_{\alpha\mu\nu}+\tilde{\alpha}_{12}\hat{B}_{\nu\alpha\mu}+\tilde{\alpha}_{13} \hat{B}_{\mu\nu\alpha}+\tilde{\alpha}_{14} \hat{B}_{\alpha\nu\mu}+\tilde{\alpha}_{15} \hat{B}_{\nu\mu\alpha} +\tilde{\alpha}_{16} \hat{B}_{\mu\alpha\nu}+\tilde{\alpha}_{17}\breve{B}_{\alpha\mu\nu}+\tilde{\alpha}_{18} \breve{B}_{\mu\nu\alpha} + \tilde{\alpha}_{19}\breve{B}_{\nu\alpha\mu} \nonumber \\
		+\tilde{\alpha}_{1,10}\bar{B}_{\alpha\mu\nu}+\tilde{\alpha}_{1,11}\bar{B}_{\mu\nu\alpha}+\tilde{\alpha}_{1,12}\bar{B}_{\nu\alpha\mu}+\tilde{\alpha}_{1,13}\mathring{B}_{\alpha\mu\nu}+\tilde{\alpha}_{1,14}\mathring{B}_{\mu\nu\alpha}+\tilde{\alpha}_{1,15}\mathring{B}_{\nu\alpha\mu} \label{NNN}
		\end{gather}
		where $\tilde{\alpha}_{ij}$ are the elements of the inverse matrix $A^{-1}$. The latter is the exact and unique solution of ($\ref{eq1}$) provided that the non-degeneracy conditions ($\ref{C}$) and ($\ref{A}$) are satisfied.

	\end{proof}

	Evidently, when $B_{\alpha}=0$ and $det(A)\neq 0$ along with $det(\Gamma)\neq 0$  we have that the unique solution of $(\ref{eq1})$ is always $N_{\alpha\mu\nu}=0$. More precisely, we have the following.
	\begin{corollary*}
		If $B_{\alpha\mu\nu}=0$ and both matrices $A$ and $\Gamma$ are non-singular, then the unique solution of $(\ref{eq1})$ is $N_{\alpha\mu\nu}=0$.
	\end{corollary*}
	
	\section{Conclusions}
	We have considered the most general  linear tensor equation of a rank $3$ tensor $N$ in terms of a given source $B$  in $4$ dimensions. The latter is a $30$ parameter tensor equation as given by ($\ref{eq1}$). By following a step by step procedure and given two rather general non-degeneracy conditions among the parameters we provided the unique and exact solution of the components $N_{\alpha\mu\nu}$ in terms of the known ones $B_{\alpha\mu\nu}$, its dualizations and its contractions. The solution is given by the expression ($\ref{NNN}$). An immediate conclusion is that, provided that the  the non-degeneracy conditions hold, in the absence of sources (i.e. when $B_{\alpha\mu\nu}=0$), $N_{\alpha\mu\nu}=0$ is the only (unique) solution of the $30$ parameter tensor equation. As a final remark, let us note that our result finds a natural application in Metric-Affine Theories of gravitation but it could just as well be applied to other physical situations.

	\section{Acknowledgments}	This research is co-financed by Greece and the European Union (European Social Fund- ESF) through the
	Operational Programme 'Human Resources Development, Education and Lifelong Learning' in the context
	of the project “Reinforcement of Postdoctoral Researchers-2nd Cycle” (MIS-5033021), implemented by the
	State Scholarships Foundation (IKY).

	\appendix
	
	\section{The $\gamma_{ij}$'s}
	The relations between the elements of $\Gamma$ and the $30$ initial parameters read
	\begin{gather}
	\gamma_{11}=a_{1}+a_{3}+a_{71}+4 a_{81}+a_{91}\;\;, \;\; \gamma_{12}=a_{2}+a_{4}+a_{72}+4 a_{82}+ a_{92}\;\;, \;\; \gamma_{13}=a_{5}+a_{6}+a_{73}+4 a_{83}+a_{93} \nonumber \\
	\gamma_{21}=a_{2}+a_{5}+4 a_{71}+ a_{81}+a_{91}\;\;, \;\; \gamma_{22}=a_{1}+a_{6}+4 a_{72}+a_{82}+ a_{92}\;\;, \;\; \gamma_{23}=a_{3}+a_{4}+4 a_{73}+ a_{83}+a_{93} \nonumber \\		
	\gamma_{31}=a_{5}+a_{6}+a_{71}+ a_{81}+ 4 a_{91}\;\;, \;\; \gamma_{32}=a_{3}+a_{4}+a_{72}+ a_{82}+4 a_{92}\;\;, \;\; \gamma_{31}=a_{1}+a_{2}+a_{73}+ a_{83}+4 a_{93} \nonumber \\
	\gamma_{41}=-2 (-1)^{s}\Big( b_{21}+b_{22}+b_{23}+b_{31}+b_{32}+b_{33} -3 b_{1}\Big)\; \; \; ,\; \;\; \;  \gamma_{42}=-2 (-1)^{s}\Big( b_{11}+b_{12}+b_{13}-b_{31}-b_{32}-b_{33} -3 b_{2}\Big) \nonumber \\
	\gamma_{43}=2 (-1)^{s}\Big( b_{11}+b_{12}+b_{13}+b_{21}+b_{22}+b_{23}+3 b_{3} \Big)\; \;\;,\;\; \; \; \gamma_{44}=a_{1}+a_{2}+a_{3}-a_{4}-a_{5}-a_{6} \nonumber \\
	\gamma_{14}=c_{1}+4 c_{2}+c_{3}-b_{12}-b_{13}-b_{21}-b_{23}-b_{31}-b_{33}\;\;, \;\; \gamma_{24}=4 c_{1}+c_{2}+ c_{3}+b_{11}+b_{22}-b_{12}-b_{21}-b_{31}-b_{32}\nonumber \\
	\gamma_{34}=c_{1}+c_{2}+4 c_{3}+b_{12}-b_{13}-b_{22}+b_{23}+b_{32}-b_{33}
	\end{gather}
These are the elements of the $4 \times 4$ matrix $\Gamma$.	
\section{The $\alpha_{ij}$'s }
Placing equations $(\ref{1})-(\ref{6})$, $(\ref{e1})-(\ref{e3})$ and $(\ref{ee1})-(\ref{ee6})$ one after another in that exact order, the coefficients of the combinations of the tensor $N$ in each equation represent the rows of matrix $A$ (with the system of equations written in matrix form $A \mathcal{N}=\mathcal{B}$). That is, the coefficients appearing in $(\ref{1})$ are the first row elements of $A$, namely
\begin{gather}
\alpha_{11}=a_{1}\;\;, \;\; \alpha_{12}=a_{2}\;\;, \;\;\alpha_{13}=a_{3}\;\;, \;\;\alpha_{14}=a_{4}\;\;, \;\;\alpha_{15}=a_{5}\;\;, \;\;\alpha_{16}=a_{6}\;\;, \;\; 
\alpha_{17}=b_{11}\;\;, \;\; \alpha_{18}=b_{12}, \;\; \alpha_{19}=b_{13}\nonumber \\
\alpha_{1,10}=b_{21}\;\;, \;\; \alpha_{1,11}=b_{22}, \;\; \alpha_{1,12}=b_{23} \;\;, \;\; \alpha_{1,13}=b_{31} \;\; , \;\; \alpha_{1,14}=b_{32} \;\;,\;\; \alpha_{1,15}=b_{33}
\end{gather}
and of course the same goes for every other row with the last one being (the coefficients of the row corresponding to ($\ref{ee6}$))
\begin{gather}
\alpha_{15,1}=(-1)^{s}(2 b_{11}-b_{12}-b_{13}) \;\;, \;\; \alpha_{15,2}=(-1)^{s}(2 b_{31}-b_{32}-b_{33}) \;\;, \;\; \alpha_{15,3}=-(-1)^{s}(2 b_{21}-b_{22}-b_{23}) \nonumber \\
\alpha_{15,4}=(-1)^{s}(2 b_{21}-b_{22}-b_{23}) \;\;, \;\; \alpha_{15,5}=-(-1)^{s}(2 b_{31}-b_{32}-b_{33}) \;\;, \;\; \alpha_{15,6}=-(-1)^{s}(2 b_{11}-b_{12}-b_{13}) \nonumber \\
\alpha_{15,7}=0\;\;, \;\; \alpha_{15,8}=a_{3}-a_{5} \;\;, \;\; \alpha_{15,9}=0 \;\;, \;\; \alpha_{15,10}=0 \;\;, \;\; \alpha_{15,11}=a_{6}-a_{2} \;\;, \;\; \alpha_{15,12}=0 \nonumber \\
\alpha_{15,3}=0\;\;, \;\; \alpha_{15,14}=a_{1}-a_{4} \;\;, \;\; \alpha_{15,15}=0 
\end{gather}

	\section{Details on the derivations}
	Let us now provide some additional information regarding the calculations we used for the proof. Firstly, using the antisymmetry of $\varepsilon$ and  of $M^{(i)}_{\alpha\mu\nu}$'s in their last pair of indices we trivially find
	\beq
	\varepsilon^{\alpha\mu}_{\;\;\;\;\beta\gamma}M^{(i)}_{\alpha\mu\nu}=\varepsilon^{\alpha\mu}_{\;\;\;\;\beta\gamma}M^{(i)}_{\mu\nu\alpha} \;,\;\;\;\; \forall \; i=1,2,3
	\eeq

	 Continuing, we compute
	\beq
	\varepsilon^{\alpha\mu}_{\;\;\;\;\beta\gamma}M^{(1)}_{\alpha\mu\nu}=2(-1)^{s+1}\Big[ \Big(N^{(2)}_{[\beta}-N^{(3)}_{[\beta}\Big) g_{\gamma]\nu}+N_{[\beta\gamma]\nu} \Big]
	\eeq
		\beq
	\varepsilon^{\alpha\mu}_{\;\;\;\;\beta\gamma}M^{(2)}_{\alpha\mu\nu}=2(-1)^{s+1}\Big[ \Big(N^{(1)}_{[\beta}-N^{(3)}_{[\beta}\Big) g_{\gamma]\nu}+N_{[\beta|\nu|\gamma]} \Big]
	\eeq
		\beq
	\varepsilon^{\alpha\mu}_{\;\;\;\;\beta\gamma}M^{(3)}_{\alpha\mu\nu}=2(-1)^{s+1}\Big[ \Big(N^{(1)}_{[\beta}-N^{(2)}_{[\beta}\Big) g_{\gamma]\nu}+N_{\nu[\beta\gamma]} \Big]
	\eeq
	The identity $\varepsilon_{\mu\alpha\beta\gamma}\varepsilon^{\mu\kappa\lambda\rho}=(-1)^{s}3! \delta_{[\alpha}^{\kappa}\delta_{\beta}^{\lambda}\delta_{\gamma]}^{\rho}$ was of great use here. In the same manner we find
	\beq
	\varepsilon^{\alpha\mu}_{\;\;\;\;\beta\gamma}M^{(1)}_{\nu\alpha\mu}=4(-1)^{s} N_{[\beta\gamma]\nu}
	\eeq
	\beq
	\varepsilon^{\alpha\mu}_{\;\;\;\;\beta\gamma}M^{(2)}_{\nu\alpha\mu}=4(-1)^{s} N_{[\beta|\nu|\gamma]}
	\eeq
	\beq
	\varepsilon^{\alpha\mu}_{\;\;\;\;\beta\gamma}M^{(3)}_{\nu\alpha\mu}=4(-1)^{s} N_{\nu[\beta\gamma]}
	\eeq
which we useful in proving out our Theorem.

	\bibliographystyle{unsrt}
	\bibliography{ref}

	\end{document}